\documentclass[onecolumn,11pt,draft=false]{IEEEtran}
\usepackage{mathtools,color}
\usepackage{hyperref}
\usepackage{enumitem}
\usepackage{balance}
\usepackage{setspace}
\usepackage{latexsym}
\usepackage{lipsum}
\usepackage{amsfonts}
\usepackage{amssymb,amsmath,bbm}
\usepackage{cleveref}
\usepackage{tikz}
\usepackage{bbm, dsfont}
\usetikzlibrary{fit,positioning}
\newtheorem{example}{Example}[section]
\newtheorem{remark}{Remark}
\newtheorem{lemma}{Lemma}
\newtheorem{theorem}{Theorem}

\newtheorem{corollary}{Corollary}
\newtheorem{proof}{Proof}
\newtheorem{definition}{Definition}
\usepackage{caption}
\usepackage{subcaption}

\newcommand{\oprocendsymbol}{\hbox{$\bullet$}}
\newcommand{\oprocend}{\relax\ifmmode\else\unskip\hfill\fi\oprocendsymbol}
\allowdisplaybreaks

\begin{document}
\title{Generating Preferential Attachment Graphs \\ via a P\'{o}lya Urn with Expanding Colors}
\author{Somya Singh$^{1}$,
         Fady Alajaji$^{2}$ and
         Bahman Gharesifard$^{3}$}
  \renewcommand{\thefootnote}{\fnsymbol{footnote}}
  \footnotetext{\\
  $^1$ICTEAM Institute, UCL, Louvain-la-Neuve, Belgium, email:somya.singh@uclouvain.be \\
  $^2$Department of Mathematics and Statistics, Queen's University, Kingston, Canada \\
  $^3$Electrical and Computer Engineering Department, UCLA, Los Angeles, USA
  }
\maketitle

\begin{abstract}
We introduce a novel preferential attachment model using the draw variables of a modified P\'{o}lya urn with an expanding number of colors, notably capable of modeling influential opinions (in terms of vertices of high degree) as the graph evolves. 
Similar to the Barab\'{a}si-Albert model, the generated graph grows in size by one vertex at each time instance; in contrast however, each vertex of the graph is uniquely characterized by a color, which is represented by a ball color in the P\'{o}lya urn.  More specifically at each time step, we draw a ball from the urn and return it to the urn along with a number of reinforcing balls of the same color; we also add another ball of a new color to the urn.
We then construct an edge between the new vertex (corresponding to the new color) and the existing vertex whose color ball is drawn. Using color-coded vertices in conjunction with the time-varying reinforcing parameter allows for vertices added (born) later in the process to potentially attain a high degree in a way that is not captured in the Barab\'{a}si-Albert model. We study the degree count of the vertices by analyzing the draw vectors of the underlying stochastic process. In particular, we establish the probability distribution of the random variable counting the number of draws of a given color which determines the degree of the vertex corresponding to that color in the graph. We further provide simulation results presenting a comparison between our model and the Barab\'{a}si-Albert network. 

\end{abstract}

\begin{spacing}{1.59}
\section{Introduction}
Preferential attachment graphs are an important class of randomly generated graphs which are often used to capture the ``rich gets richer'' phenomenon. This class of random graphs 
has been widely studied within the areas of statistical mechanics~\cite{AKH-MW:01,CT-RO:22}, network science~\cite{GC-BS:04}, probability theory~\cite{WK:05,MZ-AS:22} and game theory~\cite{FCS-JMP:05}. One of the most popular model of a preferential attachment graph is the so-called Barab\'{a}si-Albert model~\cite{ALB-RA:99}, which has since been modified in a variety of ways~\cite{XZ-ZG-BW:12,ML-GW-TC:06,TFAA-GAA-AFM-RSF-FWSL:21}. Various other models have been devised thereafter to  generate preferential attachment graphs; for example in~\cite{NB-CB-JTC-RMD-RDK:05} the growth of the random graph is competition based. Given a graph at a certain time step, the new vertex attaches itself to an existing vertex which ends up minimizing a certain cost function. For a vertex, this cost function depends on its centrality and distance from the root ensuring that the vertices with higher degrees have lower cost functions. In~\cite{PLK-SR-FL:00}, a continuous-time equation governing the number of vertices with degree $k$ is formulated to study citation networks. In~\cite{ADF-AMF-JV:06}, a randomly growing graph algorithm that combines the features of a geometric random graph and a preferential attachment graph is analysed. In~\cite{AC-VDPS-FC:06}, properties of Wikipedia are studied by representing topics as vertices and hyperlinks between them as edges. Several preferential attachment hypergraphs (i.e., graphs in which an edge can join any number of vertices) generating models have also been devised in the literature \cite{CA-ZL-YN-DP:19,FG-NN-TT-MS:22,MI-TP-HS:22}.  

Our main objective in this paper is to introduce a new preferential attachment graph generating algorithm using a modified P\'{o}lya urn model.
In the classical two-color P\'{o}lya urn  model, at each time instant $t$, a ball is drawn from the urn and returned to the urn along with one ball of the color drawn~\cite{GP:30,HM:08,RP:07}. In this sense, the P\'{o}lya urn model is a suitable reinforcement process for modeling preferential phenomena. In particular, in the context of randomly growing graphs, the reinforcements favor a high degree vertex in getting an even higher degree as the network grows. 
Indeed,  P\'{o}lya urn  models have been widely used to model reinforcements, for instance, in communication channels~\cite{polya-ch}, image segmentation~\cite{banerjee1999image}, social and epidemic networks~\cite{CA-HD-BK:20,JYK-HHJ:10,RT-JPO-JS-JH-KK:06,hayhoe18,SS-FA-BG:22,siam22,consensus22,AJ-AM-EM-RS:23,XH-JL-HD-XJ:23}, citation networks~\cite{HJ-ZN-ALB:03,SM:10,MEJN:01}, actor collaborations~\cite{RA-ALB:00,HJ-ZN-ALB:03}, mechanics \cite{RA-ALB:02} and Polymer formations \cite{KRB-RK-HM:22}. Naturally, P\'{o}lya urns have been used to model  preferential attachment graphs in the literature. For example,  in \cite{AC-CC-ML:13} a generalized P\'{o}lya urn process is used to devise a preferential attachment graph generating algorithm (refer to~\cite{FC-SH-DJ:03,RIO:09} for a detailed description of this generalized P\'{o}lya urn process.) In \cite{NB-CB-JT-AS:14}, a preferential attachment type multi-graph (i.e., a graph that can have more than one edge between a pair of vertices) is constructed using different variations of the P\'{o}lya urn process which was used to study the spread of viruses on the internet in \cite{NB-CB-JC-AS:05}. In \cite{RM-GL:19}, the probability distribution of weights on the edges of a fixed network is established via the draws of the classical two-color P\'{o}lya urn. This setup ensures that the more two vertices have interacted in the past, the more likely they are to interact in the future. More elaboration on the similarities between P\'{o}lya urns and preferential attachment graphs is given in the survey in~\cite{RP:07}.

The Barab\'{a}si-Albert model is well known to
exhibit a power law distribution as the number of vertices becomes sufficiently large, given by $p(k) \sim k^{-3}$, where $p(k)$ is the probability of randomly selecting a vertex with degree $k$ in the network. Despite the fact that this power law can be used to study various properties of the Barab\'{a}si-Albert model such as  Hirsch index distribution and the clustering coefficient, see~\cite{ALB:09,BB-OR:04,RA-ALB:00} for definitions, the likelihood of vertices gaining new edges is solely determined by their degree. This is not realistic, when modeling scenarios where newly added individuals are accompanied with {\em impactful ideas that can lead to rapid or disruptive influence, regardless of their initially low degree}.

Motivated by the above mentioned shortcoming of the Barab\'{a}si-Albert model, in this paper, we construct randomly growing undirected graphs using the draw variables of a single modified P\'{o}lya urn with an expanding number of colors. The P\'{o}lya process is modified in the sense that at each time instant, not only is a ball drawn and returned to the urn along with additional reinforcing balls of the same color, but another ball of a new color is also added to the urn. This new color corresponds to a new vertex which is added to the graph at this time instant. More specifically, the network is generated by associating each incoming vertex to the new color ball added after each draw and by attaching it to the existing vertex represented by the drawn color. The number of colors in the urn grows without bound with the number of draws, and the generated network has a preferential attachment property as the vertices corresponding to dominant colors (i.e., colors in the urn with a large number of balls) are more likely to attract newly formed vertices as their neighbors. The resulting preferential attachment growing graph is thus constructed via a P\'{o}lya urn; this enables us to track and characterize the degree count of individual vertices in the network through the draw variables of their corresponding colors giving each vertex a unique identity which is absent in the Barab\'{a}si-Albert model. Moreover, using an expanding color P\'{o}lya urn with each vertex uniquely corresponding to a ball color, sets our model apart from other models in the literature~\cite{AC-CC-ML:13,RM-GL:19,NB-CB-JC-AS:05} which use either the classical two-color P\'{o}lya urns or a generalized version of P\'{o}lya urns (with finitely many colors) for generating preferential attachment graphs. Indeed, the draw variables of the P\'{o}lya urn capture the entire structure of the generated graph and hence it is enough to study the behaviour of these draw variables to understand the properties of the graph. Moreover, we use an extra time-varying parameter to set the number of balls (not necessarily an integer) added to the P\'{o}lya urn to reinforce the color of the drawn ball at each time instant. The time-varying nature of this reinforcement parameter allows us for any given vertex to tweak the likelihood of amplifying or dampening its degree growth depending on the time at which it was introduced in the network (i.e., its birth time). This feature can be used to regulate the dominance (in terms of gaining edges) of high degree vertices over low degree ones in the generated preferential attachment graph. Therefore, unlike the Barab\'{a}si-Albert algorithm and the aforementioned models, our model can be used to generate random networks with a variety of degree distributions catering to a wide range of real-world growing networks other than power law distributions, including the generation of networks where recently formed vertices can play a
disruptive role. Furthermore, we observe that in the special case of using a reinforcement parameter equal to $1$, our model is actually equivalent to the Barab\'{a}si-Albert algorithm (except for the initialization step).

This paper is organized as follows. In Section \ref{sec:model}, we describe our model for constructing preferential attachment type graphs using a modified P\'{o}lya urn with expanding colors. We determine the P\'{o}lya urn's composition in terms of its draw random variables and use it to establish the conditional distribution of the urn's draw vector at a given time instant given all past draw vectors. In Section \ref{sec:degree}, we define a counting random variable which tracks the degree of vertices in the graph as it evolves and analytically derive its probability distribution. We further verify that the resulting distribution expression of the counting random variable defines a legitimate probability mass function over its support set. In Section \ref{sec:simulations}, we present via simulations a detailed comparison of our model with the Barab\'{a}si-Albert model by plotting the degree distribution and vertices average birth time  for various choices of reinforcement parameters. We discuss the advantages of generating preferential attachment networks through our model over the Barab\'{a}si-Albert model. Finally, conclusions are stated in Section \ref{sec:conclusions}.

\section{The Model}\label{sec:model}
 We construct a sequence of undirected graphs $\mathcal{G}_{t}$, where $t\geq 0$ denotes the time index, using a P\'{o}lya  reinforcement process. We start with $\mathcal{G}_{0} = (V_{0},\mathcal{E}_{0})$, where the initial vertex and the edge set are respectively, $V_{0}=\{c_{1}\}$ and $\mathcal{E}_{0} =\{(1,1)\}$, i.e., a self-loop on vertex $1$.
At each time step $t\geq 1$, a new vertex enters the graph and forms an edge with one of the existing vertices. The latter vertex is selected according to the draw variable of a P\'{o}lya urn with an expanding number of colors as follows:  
\begin{itemize}
    \item At time $t=0$, the P\'{o}lya urn consists of a single ball of color $c_{1}$.    
    \item At each time instant $t\geq 1$, we draw a ball and return it to the urn along with $\Delta_{t}>0$ additional (reinforcing) balls of the same color. We also add a ball of a new color $c_{t+1}$. We then introduce a new vertex to the graph $\mathcal{G}_{t-1}$ (corresponding to the color $c_{t+1}$) and connect it with the vertex whose color ball is drawn at time $t$. This results in the newly formed graph $\mathcal{G}_{t}$. Note that at time $t=0$, the urn consist of only one $c_{1}$ color ball. Hence, the draw variable at time $t=1$ is deterministic and corresponds to drawing a $c_{1}$ color ball.
\end{itemize}
At any given time instant $t$, we define the draw random vector 
\[\mathbf{Z}_{t}:=(Z_{1,t},Z_{2,t},\cdots,Z_{t,t})\] 
of length $t$, where 
\begin{align}\label{draw-mechanism}
Z_{j,t}= \begin{cases}
1  & \textrm{if a $c_{j}$ color ball is drawn at time $t$}\\
0 \quad & \textrm{otherwise}
\end{cases}
\quad \quad \textrm{for } 1\leq j \leq t.
\end{align}
The vector $\textbf{Z}_{t}$ is a standard unit vector for all time instances $t\geq 1$, and since at time $t=1$ there is only $c_{1}$ color ball present in the urn, $\textbf{Z}_{1}= Z_{1,1} =1$. We denote the ``composition'' of the P\'{o}lya urn at any given time instant $t$ by the random vector 
\[\mathbf{U}_{t} := (U_{1,t},U_{2,t},\cdots,U_{t+1,t}),\] 
where
\begin{align}\label{eqn:ratio_def}
    U_{j,t} = \frac{\textrm{Number of balls of color}\hspace{0.1cm} c_{j} \hspace{0.1cm} \textrm
    {in the urn at time}\hspace{0.1cm} t}{\textrm{Total number of balls in the urn at time}\hspace{0.1cm} t}, 
\end{align}
for $1 \leq j \leq t+1$. In the following lemma, we express the vector $\mathbf{U}_{t}$ in terms of the draw variables.
\begin{lemma}\label{lemma:ratio}
Given $t\geq 0$, $\textbf{U}_{t}$ is given by
\begin{align}\label{eqn:ratio_vector}
\textbf{U}_{t} = \frac{1}{1+t+\sum\limits_{k=1}^{t}\Delta_{k}}\Big(1+ \sum\limits_{n=1}^{t}\Delta_{n}Z_{1,n}, 1 + \sum\limits_{n=2}^{t}\Delta_{n}Z_{2,n}, \cdots, 1+ \sum\limits_{n=t-1}^{t}\Delta_{n}Z_{t-1,n}, 1 + \Delta_{t}Z_{t,t}, 1\Big)
\end{align}
almost surely.\footnote{All identities involving random variables or vectors are (implicitly) understood to hold almost surely.} 
\end{lemma}
\begin{proof}
To compute the ratio in~\eqref{eqn:ratio_def}, recall that at time $n=0$, we have one ball in the urn (this ball is of color $c_{1}$) and for each time instant $n \geq 1$, we add $\Delta_{n}+1$ balls to the urn ($\Delta_{n}$ of the color drawn and $1$ of the new color $c_{n+1}$). Hence the total number of balls in the urn at time $t$ is given by $1+\sum_{n=1}^{t}(\Delta_{n}+1)$.

To determine the number of balls of color $c_{j}$ in the urn after the $t$th draw, we note that the first $c_{j}$ color ball is added to the urn at time $j-1$. After that, at every time instant $n$ (where $j\leq n \leq t$) at which a $c_{j}$ color ball is drawn, we add $\Delta_{n}$ balls of $c_{j}$ color to the urn. Hence, the number of balls of color $c_{j}$ in the urn at time $t$ is equal to $1+\sum\limits_{n=j}^{t}\Delta_{n}Z_{j,n}$. Therefore, the ratio of color $c_{j}$ balls in the urn at time $t$ in \eqref{eqn:ratio_def} is given by:
\begin{align}\label{eqn:entryofU}
U_{j,t} = \frac{1+\sum_{n=j}^{t}\Delta_{n}Z_{j,n}}{1+t+\sum\limits_{k=1}^{t}\Delta_{k}} \hspace{0.5cm} \textrm{for }1\leq j \leq t+1,
\end{align}
which yields \eqref{eqn:ratio_vector}.
\end{proof}
\begin{remark}
As expected, the sum of the components of $\mathbf{U}_{t}$ in \eqref{eqn:ratio_vector} is one for all $t\geq0$. To see this we first note that, for $t=0$, $\mathbf{U}_{0}=U_{1,0} = 1$. For any time $t\geq 1$, we have the following from \eqref{eqn:ratio_vector}:
\begin{align}\label{eqn:sum_of_ratios}
\sum\limits_{j=1}^{t+1}U_{j,t} &= \frac{1}{1+t+\sum\limits_{k=1}^{t}\Delta_{k}}((t+1) + \sum\limits_{n=1}^{t}\Delta_{n}Z_{1,n} + \sum\limits_{n=2}^{t}\Delta_{n}Z_{2,n} + \cdots + \sum\limits_{n=t-1}^{t}\Delta_{n}Z_{t-1,n}+ \Delta_{t}Z_{t,t})\nonumber\\
& =  \frac{1}{1+t+\sum\limits_{k=1}^{t}\Delta_{k}}((t+1) +\sum\limits_{i=1}^{t}\sum\limits_{n=i}^{t}\Delta_{n}Z_{i,t})\nonumber\\
& =  \frac{1}{1+t+\sum\limits_{k=1}^{t}\Delta_{k}}((t+1) +\sum\limits_{n=1}^{t}\Delta_{n}\sum\limits_{i=n}^{t}Z_{i,t})
\end{align}
but since $\mathbf{Z}_{t}$ is a standard unit vector for all $t \geq 1$, the right-hand side of \eqref{eqn:sum_of_ratios} simplifies as follows:
\begin{align*}
\sum\limits_{j=1}^{t+1}U_{j,t} = \frac{1}{1+t+\sum\limits_{k=1}^{t}\Delta_{k}}((t+1) + \sum\limits_{n=1}^{t}\Delta_{n}) = 1. \hspace{6cm}   
\end{align*}
\end{remark}
 An illustration of our model is given in  Figure \ref{fig:urn_graph_picture} where we show a sample path of the random vectors $\mathbf{U}_{t}$ and $\mathbf{Z}_{t}$, for $ t\leq 3$ and with $\Delta_{t}=2$. 
\begin{figure}
\centering
\includegraphics[scale=0.24]{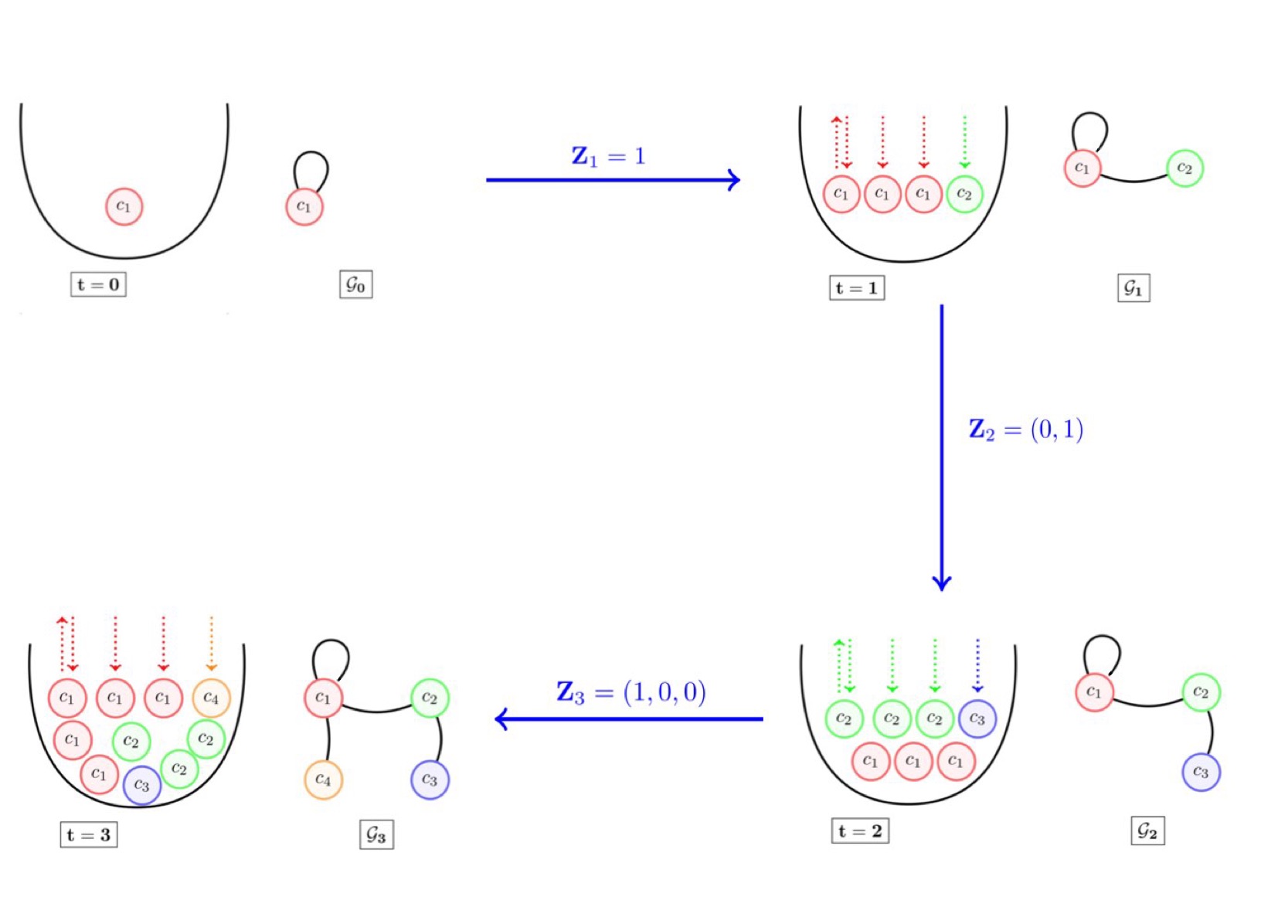}
\caption{We illustrate a sample path for constructing a preferential attachment graph using an expanding color P\'{o}lya urn with $\Delta_{t}=2$.  For $t=0$, the urn has only one ball of color $c_{1}$. This urn corresponds to $\mathcal{G}_{0}$ and $\mathbf{U}_{0}=U_{1,0}=1$. For $t=1$, the $c_{1}$ color ball is drawn from and returned to the urn (i.e., $\mathbf{Z}_{1}=Z_{1,1}=1$). Two additional $c_{1}$ color balls are added to the urn along with a new $c_{2}$ color ball and so $\mathbf{U}_{1}=(3/4,1/4)$. For $t=2$, a $c_{2}$ color ball is drawn from and returned to the urn (i.e., $\mathbf{Z}_{2}=(0,1)$). Two additional $c_{2}$ color balls  are added to the urn along with a new $c_{3}$ color ball; hence $\mathbf{U}_{2}=(3/7,3/7,1/7)$. For $t=3$, a $c_{1}$ color ball is drawn from and returned to the urn (i.e., $\mathbf{Z}_{3}=(1,0,0)$). Two additional $c_{1}$ color balls are added along with a new $c_{4}$ color ball; thus $\mathbf{U}_{3}=(5/10,3/10,1/10,1/10)$.}
\label{fig:urn_graph_picture}
\end{figure}
 We further write the conditional probabilities of the draw variables given the past. More specifically, for $1\leq j\leq t$, using \eqref{eqn:entryofU}, we have that
\begin{align}\label{eqn:cond_prob}
P( \textbf{Z}_{t} =\textbf{e}_{j,t}\hspace{0.05cm}|\hspace{0.05cm}\textbf{Z}_{t-1},\textbf{Z}_{t-2},\cdots,\textbf{Z}_{1})
&= P(Z_{j,t}= 1\hspace{0.05cm}|\hspace{0.05cm}\textbf{Z}_{t-1},\textbf{Z}_{t-2},\cdots,\textbf{Z}_{1}) \nonumber\\
&= P(\textrm{a $c_{j}$ color ball is drawn at time $t$}\hspace{0.05cm}|\hspace{0.05cm}\textbf{Z}_{t-1},\textbf{Z}_{t-2},\cdots,\textbf{Z}_{1})\nonumber\\
&= U_{j,t-1} = \frac{1 + \sum_{n=j}^{t-1}\Delta_{n}Z_{j,n}}{1+(t-1)+\sum\limits_{k=1}^{t-1}\Delta_{k}},
\end{align}
where $\mathbf{e}_{j,t}$ represents a standard unit vector of length $t$ whose $j$th component is $1$. Considering the case where $j=t$ in \eqref{eqn:cond_prob} and the convention that $\sum_{n=t}^{t-1}\Delta_{n}Z_{t,n}=0$, we obtain 
\[U_{t,t-1} = \frac{1 + \sum_{n=t}^{t-1}\Delta_{n}Z_{t,n}}{t+\sum\limits_{k=1}^{t-1}\Delta_{k}}=\frac{1}{t+\sum\limits_{k=1}^{t-1}\Delta_{k}} = P(Z_{t,t}=1)\]
and hence
\begin{align}\label{eqn:cond_one}  
P(Z_{t,t}=1\hspace{0.05cm}|\hspace{0.05cm}\mathbf{Z}_{t-1},\mathbf{Z}_{t-2},\cdots,\mathbf{Z}_{1}) = P(Z_{t,t}=1)=\frac{1}{t+\sum_{k=1}^{t-1}\Delta_{k}},
\end{align}
i.e.,  the conditional probability of drawing a ball of color $c_{t}$ at time $t$ equals the marginal probability of drawing a ball of color $c_{t}$ at time $t$. Similarly, we have that
\begin{align}\label{eqn:cond_zero}  
P(Z_{t,t}=0\hspace{0.05cm}|\hspace{0.05cm}\mathbf{Z}_{t-1},\mathbf{Z}_{t-2},\cdots,\mathbf{Z}_{1}) = P(Z_{t,t}=0)=\frac{t-1+\sum_{n=1}^{t-1}\Delta_{n}}{t+\sum_{k=1}^{t-1}\Delta_{k}},
\end{align}
implying that $Z_{t,t}$ is independent of the random vectors $\{\mathbf{Z}_{t-1},\mathbf{Z}_{t-2},\cdots,\mathbf{Z}_{1}\}$. More generally, we obtain the marginal probability for the random variable $Z_{j,t}$ for any $1\leq j \leq t$ by taking expectation on both sides in \eqref{eqn:cond_prob} with respect to the random vectors $\textbf{Z}_{t-1},\textbf{Z}_{t-2},\cdots, \textbf{Z}_{1}$ as follows:
\begin{align}\label{eqn:marginal_prob}
P(Z_{j,t}=1) = E(U_{j,t-1})= \frac{1 + \sum_{n=j}^{t-1}\Delta_{n}P(Z_{j,n}=1)}{t+\sum_{k=1}^{t-1}\Delta_{k}} = 1-P(\hspace{0.03cm}Z_{j,t}=0) \quad \quad \textrm{for} \quad 1\leq j \leq t.
\end{align}
For $j=t$, the formula in \eqref{eqn:marginal_prob} for $P(Z_{t,t}=1)$ reduces to \eqref{eqn:cond_one}, but for $j<t$, the formula for $P(Z_{j,t}=1)$ is a recursive function of the marginal probabilities of past draw variables: $P(Z_{j,1}=1),\cdots,P(Z_{j,t-1}=1)$.

We further note that, for graph $\mathcal{G}_{t}$, the edge between the new vertex to one of the existing vertices in $\mathcal{G}_{t-1}$ is made using the realization of the draw vector $\mathbf{Z}_{t}$. Using \eqref{eqn:cond_prob}, we observe that the conditional probability $P(\mathbf{Z}_{t}=\mathbf{e}_{j,t}|\mathbf{Z}_{t-1},\cdots,\mathbf{Z}_{1})$ can be written in terms of the draw variables  $Z_{j,j},\cdots,Z_{j,t-1}$. Hence, all the spatial information of the graph $\mathcal{G}_{t}$ is \textit{encoded} in the sequence of random draw vectors $\{\mathbf{Z}_{1},\ldots,\mathbf{Z}_{t-1},\mathbf{Z}_{t}\}$. We illustrate this property in the following example, where we retrieve the graph $\mathcal{G}_{4}$ using $\{\mathbf{Z}_{1},\mathbf{Z}_{2},\mathbf{Z}_{3},\mathbf{Z}_{4}\}$.
\begin{example}
Consider the following realizations for the random draw vectors $\mathbf{Z}_{1}$, $\mathbf{Z}_{2}$, $\mathbf{Z}_{3}$ and $\mathbf{Z}_{4}$ for all $t\geq 1$:
\begin{align*}
\mathbf{Z}_{1} &= 1, \hspace{2cm}
\mathbf{Z}_{2} = (1,0),\\
\mathbf{Z}_{3} &= (0,1,0),\hspace{1cm}
\mathbf{Z}_{4} = (0,1,0,0).
\end{align*}
By construction, the graph $\mathcal{G}_{0}$ consists of only one vertex $c_{1}$ with a self-loop. Since we start with only one ball of color $c_{1}$ in the urn, the random variable $\mathbf{Z}_{1}= 1$ is deterministic and results in an edge drawn between the $c_{1}$ vertex and the new incoming vertex $c_{2}$. For $t=2$, since $\mathbf{Z}_{2} = (1,0)$, the new incoming vertex $c_{3}$ is connected to $c_{1}$. Similarly for $t=3$, the new vertex $c_{4}$ is connected to $c_{2}$ and finally, for $t=4$, using the value of $\mathbf{Z}_{4}$, we connect $c_{5}$ to $c_{2}$. Hence, the graph $\mathcal{G}_{4}$ is as shown in Figure \ref{encoded-graph}.
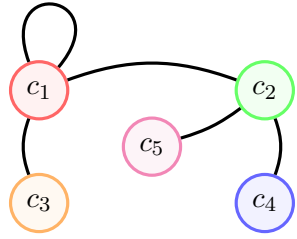
\begin{figure}
\begin{center}
\begin{tikzpicture}
[roundnode/.style={circle, draw=red!60, fill=red!5, very thick, minimum size=5mm},
greennode/.style={circle, draw=green!60, fill=green!5, very thick, minimum size=5mm},
bluenode/.style={circle, draw=blue!60, fill=blue!5, very thick, minimum size=5mm},
orangenode/.style={circle, draw=orange!60, fill=orange!5, very thick, minimum size=5mm},  
magentanode/.style={circle, draw=magenta!60, fill=magenta!5, very thick, minimum size=5mm}
]
\node[roundnode]  (1) at (1.0, -13.0) {$c_{1}$};
\node[greennode] (2)  at (4.0, -13.0) {$c_{2}$};
\node[bluenode] (4)  at (4.0, -14.5){$c_{4}$}; 
\node[orangenode](3) at (1.0,-14.5){$c_{3}$};
\node[magentanode](5) at (2.5,-13.75){$c_{5}$};
\tikzset{every loop/.style={}} 
\path[-,very thick, black]
  (1) edge [out=110, in=50, loop] node{} (1)
  (1) edge [bend left=20](2)
  (1) edge[bend right=20] (3)
  (2) edge [bend left =20] (4)
  (2) edge [ bend left =10] (5);
\end{tikzpicture}
\end{center}
\caption{An illustration of how the sequence of draw vectors $\{\mathbf{Z}_{4}=(0,1,0,0),\mathbf{Z}_{3}=(0,1,0),\mathbf{Z}_{2}=(1,0),\mathbf{Z}_{1}=1\}$ determines $\mathcal{G}_{4}$.}
\label{encoded-graph}
\medskip
\end{figure}
\end{example}
\section{Analysing the degree count of the vertices in \boldmath{$\mathcal{G}_{t}$}}\label{sec:degree}
 The goal of this section is to establish a formula for the probability distribution of degree count of a fixed vertex in the preferential attachment graph constructed via our modified P\'{o}lya process until time $t$ (including time $t$). We obtain this formula by writing the degree of a fixed vertex at time $t$ in terms of the total number of balls of color corresponding to this vertex drawn until time $t$. To this end, let random variable $N_{j,t}$ count the number of draws of color $c_{j}$ from the urn until time $t$ (including time $t$). Since by construction, for a fixed color $c_{j}$, at any time $t\geq j-1$ the degree of vertex $c_{j}$ at time $t$, denoted by $d_{j,t}$, is one more than the  number of times a $c_{j}$ color ball is drawn from the urn until time $t$; the additional one here is due to the fact that at each time instant, the new vertex which is added has degree one. For instance in Figure \ref{fig:urn_graph_picture}, the color $c_{2}$ (green) is drawn once until time $3$ and hence the degree of vertex corresponding to color $c_{2}$ in $\mathcal{G}_{3}$ is two. Therefore,
\[d_{j,t}=1+N_{j,t}  \quad \textrm{for all } 1\leq j\leq t+1,\]
where $N_{t+1,t}=0$. Also, note that the first time a color $c_{j}$ ball can be drawn is at time $j$. In the following theorem, we establish an analytical expression for the (marginal) probability mass function of random variable $N_{j,t}$. 
\begin{theorem}\label{thm:main_result}
Fix $t\geq 1$. For a color $c_{j}$, $1\leq j\leq t$, we have that 
\begin{align}\label{eqn:cdf}
&P(N_{j,t}=k)=\nonumber\\\\
&\begin{cases}\sum\limits_{(i_{1},i_{2},\cdots,i_{k})\in \mathcal{A}_{j,k}^{(t)}} \hspace{-0.8cm}\frac{\prod\limits_{a=1}^{k}\big(1+\sum\limits_{b=1}^{a-1}\Delta_{i_{b}}\big)\prod\limits_{p=j,p \notin \{i_{1},\cdots,i_{k}\}}^{t}\big((p\hspace{0.05cm}-\hspace{0.05cm}1)+ \sum\limits_{l=1,l \notin \{i_{1},\cdots,i_{k}\}}^{p-1}\Delta_{l}\big)}{\prod\limits_{n=j-1}^{t-1}\big((n+1)+\sum\limits_{m=1}^{n}\Delta_{m}\big)} &\textrm{for } 1\leq k \leq t-j+1\nonumber\\\\
\hspace{2.5cm}\frac{\prod\nolimits_{p=j}^{t}\big((p\hspace{0.05cm}-\hspace{0.05cm}1)+\sum\nolimits_{l=1}^{p-1}\Delta_{l}\big)}{\prod\nolimits_{n=j-1}^{t-1}\big((n+1)+\sum\nolimits_{m=1}^{n}\Delta_{m}\big)} &\textrm{for } k=0,
\end{cases}
\end{align}
where
\begin{align}\label{eqn:A}
\mathcal{A}_{j,k}^{(t)} = 
\begin{cases}
\{(i_{1},i_{2},\cdots,i_{k})\hspace{0.1cm}|\hspace{0.1cm} 1 = i_{1}<i_{2}<\cdots <i_{k}\leq t\} & \text{for $j=1$}\\
    \{(i_{1},i_{2},\cdots,i_{k})\hspace{0.1cm}|\hspace{0.1cm} j \leq i_{1}<i_{2}<\cdots <i_{k}\leq t\} &  \text{for $1<j\leq t$}.\\    
\end{cases}
\end{align}
\end{theorem}
\begin{proof} Note that the set $\mathcal{A}_{j,k}^{(t)}$ defined in \eqref{eqn:A} gives all possible ways in which $k$ elements can be chosen from a set of $t-j+1$ consecutive integers. In the context of our model, this set represents all possible $k$ length tuples of time instants such that a color $c_{j}$ ball is drawn at each of these time instants. For $j=1$, the first draw at $t=1$ is deterministic and hence $i_{1}=1$ for $j=1$ as given in \eqref{eqn:A}. For $1\leq k \leq t-j+1$, we have 
\begin{align}\label{eqn:cond_proof}
P(N_{j,t}=k)&=\hspace{-0.5cm}\sum\limits_{(i_{1},\cdots,i_{k})\in \mathcal{A}_{j,k}^{(t)}}\hspace{-0.5cm}P(Z_{j,j}=0,\cdots,Z_{j,i_{1}-1}=0,Z_{j,i_{1}}=1,Z_{j,i_{1}+1}=0,\cdots Z_{j,i_{2}-1}=0,Z_{j,i_{2}}=1, Z_{j,i_{2}+1}=0, \nonumber\\
&\hspace{0.5cm}\cdots,
Z_{j,i_{3}-1}=0,Z_{j,i_{3}}=1,\cdots,Z_{j,i_{k}-1}=0,Z_{j,i_{k}}=1, Z_{j,i_{k}+1}=0,\cdots, Z_{j,t}=0)\nonumber\\[10 pt]
&=\hspace{-0.5cm}\sum\limits_{(i_{1},\cdots,i_{k})\in \mathcal{A}_{j,k}^{(t)}}\hspace{-0.5cm}[P(Z_{j,t}=0\hspace{0.02cm} |\hspace{0.02cm}Z_{j,j}=0,\cdots,Z_{j,i_{1}-1}=0,Z_{j,i_{1}}=1,Z_{j,i_{1}+1}=0,\cdots, Z_{j,i_{2}-1}=0,Z_{j,i_{2}}=1,
\nonumber\\
&Z_{j,i_{2}+1}=0,\cdots,Z_{j,i_{3}-1}=0,Z_{j,i_{3}}=1,\cdots ,Z_{j,i_{k}-1}=0,Z_{j,i_{k}}=1, Z_{j,i_{k}+1}=0,\cdots, Z_{j,t-1}=0)\nonumber\\
& \times P(Z_{j,j}=0,\cdots,Z_{j,i_{1}-1}=0,Z_{j,i_{1}}=1,Z_{j,i_{1}+1}=0,\cdots,Z_{j,i_{2}-1}=0,Z_{j,i_{2}}=1,
Z_{j,i_{2}+1}=0,\nonumber\\
&\cdots,Z_{j,i_{3}-1}=0,Z_{j,i_{3}}=1,\cdots ,Z_{j,i_{k}-1}=0,Z_{j,i_{k}}=1, Z_{j,i_{k}+1}=0,\cdots, Z_{j,t-1}=0)].
\end{align}
By substituting \eqref{eqn:cond_prob} in  the conditional probability expressions in \eqref{eqn:cond_proof}, we obtain the following:
\begin{align}\label{eqn:sub_proof_1}
&\hspace{-0.5cm}P(Z_{j,t}=0\hspace{0.03cm}|\hspace{0.03cm}Z_{j,j}=0,\cdots,Z_{j,i_{1}-1}=0,Z_{j,i_{1}}=1,Z_{j,i_{1}+1}=0,\cdots ,Z_{j,i_{2}-1}=0,Z_{j,i_{2}}=1,
Z_{j,i_{2}+1}=0,\cdots \nonumber\\
&Z_{j,i_{3}-1}=0,Z_{j,i_{3}}=1,\cdots, Z_{j,i_{k}-1}=0,Z_{j,i_{k}}=1, Z_{j,i_{k}+1}=0,\cdots, Z_{j,t-1}=0) \nonumber\\ 
&= 1- \dfrac{1+ \sum_{n=j}^{t-1}\Delta_{n}Z_{j,n}}{t + \sum_{m=1}^{t-1}\Delta_{m}}\nonumber\\
&= 1- \dfrac{1+\sum_{l=1}^{k}\Delta_{i_{l}}}{t + \sum_{m=1}^{t-1}\Delta_{m}}\nonumber\\
&= \frac{(t-1) + \sum_{l=1,l\notin \{i_{1},\cdots,i_{k}\}}^{t-1}\Delta_{l}}{t + \sum_{m=1}^{t-1}\Delta_{m}}.
\end{align}
Now, substituting the conditional probability expression obtained in \eqref{eqn:sub_proof_1} in \eqref{eqn:cond_proof}, yields

\begin{align}\label{eqn:proof_second}
P(N_{j,t}=k)&=\hspace{-0.5cm}\sum\limits_{(i_{1},\cdots,i_{k})\in\mathcal{A}_{j,k}^{(t)}}\hspace{-0.2cm}\frac{((t-1) + \sum_{l=1,l\notin \{i_{1},\cdots,i_{k}\}}^{t-1}\Delta_{l})}{t + \sum_{m=1}^{t-1}\Delta_{m}}P(Z_{j,j}=0,\cdots,Z_{j,i_{1}-1}=0,Z_{j,i_{1}}=1, Z_{j,i_{1}+1}=0,\cdots \nonumber\\
&\hspace{0.5cm} \cdots,Z_{j,i_{2}-1}=0,Z_{j,i_{2}}=1,
Z_{j,i_{2}+1}=0,\cdots,Z_{j,i_{3}-1}=0, Z_{j,i_{3}}=1,\cdots,Z_{j,i_{k}-1}=0,\nonumber\\
&\hspace{0.5cm}Z_{j,i_{k}}=1, Z_{j,i_{k}+1}=0,\cdots, Z_{j,t-1}=0).
\end{align}
Similar to \eqref{eqn:cond_proof} and \eqref{eqn:proof_second}, we continue to  recursively write the joint probability as a product of conditional and marginal probabilities and substitute the expressions for the conditional probability using \eqref{eqn:sub_proof_1} as follows: 
\begin{align*}
P(N_{j,t}=k)&=\hspace{-0.5cm}\sum\limits_{(i_{1},\cdots,i_{k})\in\mathcal{A}_{j,k}^{(t)}}\hspace{-0.1cm}\Bigg(\frac{(t-1) + \sum_{l=1,l\notin \{i_{1},\cdots,i_{k}\}}^{t-1}\Delta_{l}}{t + \sum_{m=1}^{t-1}\Delta_{m}}\Bigg)\Bigg(\frac{(t-2) + \sum_{l=1,l\notin \{i_{1},\cdots,i_{k}\}}^{t-2}\Delta_{l}}{t-1 + \sum_{m=1}^{t-2}\Delta_{m}}\Bigg)P(Z_{j,j}=0,\cdots\nonumber\\
&\hspace{0.7cm}Z_{j,i_{1}-1}=0,Z_{j,i_{1}}=1,Z_{j,i_{1}+1}=0,\cdots ,Z_{j,i_{2}-1}=0,Z_{j,i_{2}}=1,
Z_{j,i_{2}+1}=0,\cdots,\nonumber\\
&\hspace{0.7cm}Z_{j,i_{3}-1}=0,Z_{j,i_{3}}=1,\cdots ,Z_{j,i_{k}-1}=0,Z_{j,i_{k}}=1,Z_{j,i_{k}+1}=0,\cdots, Z_{j,t-2}=0)\nonumber\\
&\hspace{0.7cm}\vdots \nonumber\\
&=\hspace{-0.5cm}\sum\limits_{(i_{1},\cdots,i_{k})\in\mathcal{A}_{j,k}^{(t)}}\Bigg[\Bigg(\frac{(t-1) + \sum_{l=1,l\notin \{i_{1},\cdots,i_{k}\}}^{t-1}\Delta_{l}}{t + \sum_{m=1}^{t-1}\Delta_{m}}\Bigg)\Bigg(\frac{(t-2) + \sum_{l=1,l\notin\{i_{1},\cdots,i_{k}\}}^{t-2}\Delta_{l}}{t-1 + \sum_{m=1}^{t-2}\Delta_{m}}\Bigg)\cdots\nonumber\\[8pt]
&\hspace{0.7cm}\Bigg(\frac{i_{k} + \sum_{l=1,l\notin\{i_{1},\cdots,i_{k}\}}^{i_{k}}\Delta_{l}}{i_{k}+1 + \sum_{m=1}^{i_{k}}\Delta_{m}}\Bigg)\Bigg(\frac{1+\sum\nolimits_{l=1 }^{k\hspace{0.02cm}-1}\Delta_{i_{l}}}{i_{k} +\sum_{m=1}^{i_{k}-1}\Delta_{m}}\Bigg)\Bigg(\frac{i_{k}-2+ \sum_{l=1,l\notin \{i_{1},\cdots,i_{k}\}}^{i_{k}-2}\Delta_{l}}{i_{k}-1 + \sum_{m=1}^{i_{k}-2}\Delta_{m}}\Bigg) \cdots \nonumber\\[8pt]
&\hspace{0.7cm}\Bigg(\frac{i_{k-1} + \sum_{l=1,l\notin \{i_{1},\cdots,i_{k}\}}^{i_{k-1}}\Delta_{l}}{i_{k-1}+1 + \sum_{m=1}^{i_{k-1}}\Delta_{m}}\Bigg)\hspace{-0.1cm}\Bigg(\frac{1+\sum\nolimits_{l=1}^{k\hspace{0.02cm}-2}\Delta_{i_{l}}}{i_{k-1} +\sum_{m=1}^{i_{k-1}-1}\Delta_{m}}\Bigg)\hspace{-0.1cm}\Bigg(\frac{i_{k-1}-2+ \sum_{l=1,l\notin \{i_{1},\cdots,i_{k}\}}^{i_{k-1}-2}\Delta_{l}}{i_{k-1}-1 + \sum_{m=1}^{i_{k-1}-2}\Delta_{m}}\Bigg)  \nonumber\\[8pt]
&\hspace{0.2cm}\cdots\Bigg(\frac{1}{i_{1}+\sum_{m=1}^{i_{1}-1}\Delta_{m}}\Bigg)\hspace{-0.1cm}\Bigg(\frac{i_{1}-2+\sum_{l=1, l \notin \{i_{1},\cdots,i_{k}\}}^{i_{1}-2}\Delta_{l}}{i_{1}-1 + \sum_{m=1}^{i_{1}-2}\Delta_{m}}\Bigg)\hspace{-0.1cm}\cdots \hspace{-0.1cm}\Bigg(\frac{j-1 + \sum_{l=1, l \notin \{i_{1},\cdots,i_{k}\}}^{j-1}\Delta_{l}}{j + \sum_{m=1}^{j-1}\Delta_{m}}\Bigg)\hspace{-0.1cm}\Bigg]\nonumber\\[8pt]
&=\hspace{-0.5cm}\sum\limits_{(i_{1},\cdots,i_{k})\in\mathcal{A}_{j,k}^{(t)}}\Bigg[\frac{\prod_{a=1}^{k}(1+\sum_{b=1}^{a-1}\Delta_{i_{b}})}{\prod_{n=j-1}^{t-1}((n+1)+\sum_{m=1}^{n}\Delta_{m})}\bigg(\prod_{l_{1}=j-1}^{i_{1}-2}(l_{1}+\sum\nolimits_{l=1, l \notin \{i_{1},\cdots,i_{k}\}}^{l_{1}}\Delta_{l})\bigg)\nonumber\\[8pt]
&\hspace{0.7cm}\bigg(\prod_{l_{2}=i_{1}}^{i_{2}-2}(l_{2}+ \sum\nolimits_{l=1, l \notin \{i_{1},\cdots,i_{k}\}}^{l_{2}}\Delta_{l})\bigg)\bigg(\prod_{l_{3}=i_{2}}^{i_{3}-2}(l_{3}+\sum\nolimits_{l=1,l \notin \{i_{1},\cdots,i_{k}\}}^{l_{3}}\Delta_{l})\bigg)\cdots\nonumber\\[8pt]
&\hspace{0.7cm}\cdots\bigg(\prod_{l_{k}=i_{k-1}}^{i_{k}-2}(l_{k}+\sum\nolimits_{l=1,l\notin \{i_{1},\cdots,i_{k}\}}^{l_{k}}\Delta_{l})\bigg)\bigg(\prod_{l_{k+1}=i_{k}}^{t-1}(l_{k+1}+\sum\nolimits_{l=1,l\notin \{i_{1},\cdots,i_{k}\}}^{l_{k+1}}\Delta_{l})\bigg)\Bigg]\nonumber\\[8pt]
&=\sum\limits_{(i_{1},i_{2},\cdots,i_{k})\in \mathcal{A}_{j,k}^{(t)}} \frac{\prod\nolimits_{a=1}^{k}(1+\sum\nolimits_{b=1}^{a-1}\Delta_{i_{b}})\prod\nolimits_{p=j,p \notin \{i_{1},\cdots,i_{k}\}}^{t}((\hspace{0.03cm}p-1)+ \sum\nolimits_{l=1,l \notin \{i_{1},\cdots,i_{k}\}}^{p-1}\Delta_{l})}{\prod_{n=j-1}^{t-1}((n+1)+\sum\nolimits_{m=1}^{n}\Delta_{m})}.\nonumber\\[8pt]
\end{align*}
Therefore, \eqref{eqn:cdf} holds for $1\leq k \leq t-j+1$. We determine $P(N_{j,t}=0)$ as follows:
\begin{align}\label{eqn:cdf_kzero}
P(N_{j,t}=0) &= P(Z_{j,j}=0,Z_{j,j+1}=0,\cdots,Z_{j,t-1}=0,Z_{j,t}=0)\nonumber\\
&=P(Z_{j,j}=0)\prod\limits_{n=j+1}^{t}P(Z_{j,n}=0\hspace{0.05cm}|\hspace{0.05cm}Z_{j,j}=0,Z_{j,j+1}=0,\cdots,Z_{j,n-1}=0).
\end{align}
Now, using \eqref{eqn:cond_prob} for the conditional probabilities in \eqref{eqn:cdf_kzero} we obtain
\begin{align*}
P(N_{j,t}=0) & = \Big(1-\frac{1}{j + \sum\nolimits_{m=1}^{j-1}\Delta_{m}}\Big)\Big(1-\frac{1}{j+1 +  \sum\nolimits_{m=1}^{j}\Delta_{m}}\Big)\cdots\Big(1-\frac{1}{t +  \sum\nolimits_{m=1}^{t-1}\Delta_{m}}\Big)\nonumber\\
& = \frac{\prod\nolimits_{p=j}^{t}((p-1)+\sum\nolimits_{l=1}^{p-1}\Delta_{l})}{\prod\nolimits_{n=j-1}^{t-1}((n+1) + \sum\nolimits_{m=1}^{n}\Delta_{m})}.
\end{align*}
Hence \eqref{eqn:cdf} holds for $k=0$.
\end{proof}
The analytic formula obtained in \eqref{eqn:cdf} is quite involved when the reinforcement parameter $\Delta_{t}$ is time-varying. For the special case of $\Delta_{t}=\Delta$ for all $t\geq 1$, Theorem \ref{thm:main_result} reduces to the following corollary.
\begin{corollary}\label{thm:main_result_delta}
Fix $t\geq 1$. For a color $c_{j}$, $1\leq j\leq t$ and $\Delta_{n}=\Delta$ for all $n\geq 1$, the marginal probability for $N_{j,t}$ is given by:
\begin{align}\label{eqn:cdf_delta}
&P(N_{j,t}=k)=\nonumber\\\\
&\begin{cases}\sum\limits_{(i_{1},i_{2},\cdots,i_{k})\in \mathcal{A}_{j,k}^{(t)}} \hspace{-0.8cm}\frac{\prod\limits_{a=1}^{k}\big(1+(a-1)\Delta \big)\prod\limits_{p=j,p \notin \{i_{1},\cdots,i_{k}\}}^{t}\big((p\hspace{0.05cm}-\hspace{0.05cm}1)(\Delta+1)-\Delta\sum\limits_{l=1}^{k}\mathbbm{1}(i_{l}\leq p\hspace{0.05cm}-\hspace{0.05cm}1)\big)}{\prod\limits_{n=j-1}^{t-1}\big((\Delta+1)n+1\big)} & \textrm{for } 1\leq k\leq t-j+1\nonumber\\\\
\hspace{2.5cm}\frac{\prod\limits_{p=j}^{t}(p\hspace{0.05cm}-\hspace{0.05cm}1)(\Delta+1)}{\prod\limits_{n=j-1}^{t-1}\big((\Delta+1)n+1\big)} & \textrm{for } k=0,
\end{cases}
\end{align}
 where the set $\mathcal{A}_{j,k}^{(t)}$ is defined in \eqref{eqn:A} and $\mathbbm{1}(\mathcal{E})$ is the indicator function of the event $\mathcal{E}$.
\end{corollary}
The analytical expression in~\eqref{eqn:cdf_delta} can be further simplified for the case of $\Delta_{t}=1$, $t\geq 1$, as follows.
\begin{corollary}\label{thm:main_result_one}
Fix $t\geq 1$. For a color $c_{j}$, $1\leq j\leq t$ and $\Delta_{n}=1$ for all $n\geq 1$, the marginal probability for $N_{j,t}$ is given by:
\begin{align}\label{eqn:pdf_one}
P(N_{j,t}=k)=\begin{cases}\sum\limits_{(i_{1},i_{2},\cdots,i_{k})\in \mathcal{A}_{j,k}^{(t)}}\frac{\Gamma(k+1)\big(\prod\limits_{p=j}^{t}(2(p-1)\hspace{0.01cm}-\sum\limits_{l=1}^{k}\mathbbm{1}(i_{l}\leq p\hspace{0.05cm}-\hspace{0.05cm}1)\big)}{\prod\limits_{n=j-1}^{t-1}(2n+1)} & \textrm{for } 1\leq k\leq t-j+1\\\\
\hspace{2.5cm}\frac{2\Gamma(t+1)}{\Gamma(j-1)\prod\limits_{n=j-1}^{t-1}(2n+1)} & \textrm{for } k=0,
\end{cases}
\end{align}
where $\Gamma(\cdot)$ is the gamma function.
\end{corollary}
As we will see in Section~\ref{sec:simulations}, our model with $\Delta_{t}=1$ for all $t\geq 0$ has exactly the same mechanism as the Barab\'{a}si-Albert model (except for the initialization). Therefore, \eqref{eqn:pdf_one} can be used to predict the degree count of vertices in the Barab\'{a}si-Albert model for sufficiently large $t$. To illustrate~\eqref{eqn:pdf_one}, we present a simulation in Figure~\ref{fig:main_result_one} for the probability mass function of the counting random variable $N_{2,12}$ for the case of $\Delta_{t}=1$, $1\leq t\leq 12$.

\begin{figure}
    \centering
    \includegraphics[scale=0.5]{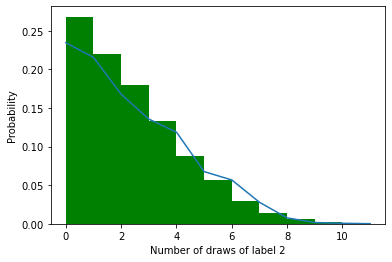}
    \caption{A simulation of the probability distribution given by \eqref{eqn:pdf_one} in Corollary \ref{thm:main_result_one} for the case of $\Delta_{t}=1$ with $1\leq t\leq 12$. 
    A normalized histogram 
    of the counting random variable $N_{2,12}$ from our model is plotted (by averaging over $1000$ simulations) and is shown to concord with the curve of~\eqref{eqn:pdf_one} (in blue).}
    \label{fig:main_result_one}
\end{figure}

Since the first time instant at which a $c_{j}$ color ball can be drawn from the P\'{o}lya urn is at time $j$, the total number of draws of a $c_{j}$ color ball till time $t$ can be at most $t-j+1$. Therefore, 
\[\sum\limits_{k=0}^{t-j+1}P(d_{j,t}=k+1)= \sum\limits_{k=0}^{t-j+1}P(N_{j,t}=k) =1,\]
which implies that $P(N_{j,t}=k)$ is a probability mass function on the support set $\{0,1,\cdots,t-j+1\}$. We next only verify that $P(N_{j,t}=k)$ obtained in \eqref{eqn:cdf_delta} does indeed sum up to one (over $k$ ranging from zero to $t-j+1$) and is hence a legitimate probability mass function. For simplicity, we focus on the case with $\Delta_{t}=\Delta$; the proof for the general case follows along similar lines. To this end, we write the set $\mathcal{A}_{j,k}^{(t)}$ as the following disjoint union:
\begin{align}\label{eqn:set}
\mathcal{A}_{j,k}^{(t)} &= \{(i_{1},i_{2},\cdots, i_{k})\hspace{0.1cm}|\hspace{0.1cm} j \leq i_{1}< i_{2}<\cdots < i_{k}\leq t\}\nonumber\\
 &= \{(i_{1},i_{2},\cdots, i_{k})\hspace{0.1cm}|\hspace{0.1cm}  j \leq i_{1}< i_{2}<\cdots < i_{k}\leq t-1\} \sqcup \mathcal{B}_{j,k}^{(t)}= \mathcal{A}_{j,k}^{(t-1)} \sqcup \mathcal{B}_{j,k}^{(t)},
\end{align}
where $\mathcal{B}_{j,k}^{(t)}: = \{(i_{1},\cdots,i_{k-1},t\hspace{0.05cm})\hspace{0.1cm}|\hspace{0.1cm} j \leq i_{1}<\cdots<i_{k-1}\leq t-1\}$. Note that 
\begin{align}\label{eqn:setB}
\mathcal{B}_{j,k}^{(t)} = \{(i_{1},\cdots,i_{k-1},t\hspace{0.05cm})\hspace{0.1cm}|\hspace{0.1cm} (i_{1},\cdots,i_{k-1})\in \mathcal{A}_{j,k-1}^{(t-1)}\}.
\end{align}
\begin{theorem}
Fix $t\geq1$. For a $c_{j}$, $1\leq j \leq t$ and $\Delta_{t}=\Delta$ for all $t\geq 1$, we have that:
\begin{align}\label{eqn:pdf}
\sum\limits_{k=0}^{t-j+1}P(N_{j,t}=k) =1.
\end{align}   
\end{theorem}
\begin{proof}
We write the left-hand side of \eqref{eqn:pdf} using \eqref{eqn:cdf_delta}:
\begin{align}\label{eqn:general_case}
&\frac{\prod\nolimits_{p=j}^{t}(p-1)(\Delta+1)}{\prod_{n=j-1}^{t-1}((\Delta+1)n+1)} \nonumber\\
&+\sum\limits_{k=1}^{t-j+1}\sum\limits_{(i_{1},\cdots,i_{k})\in \mathcal{A}_{j,k}^{(t)}}\frac{ \prod\limits_{a=1}^{k}(1+(a-1)\Delta)\prod\limits_{p=j,p \notin \{i_{1},\cdots,i_{k}\}}^{t}((\hspace{0.03cm}p-1)(\Delta+1)-\Delta\sum\nolimits_{l=1}^{k}\mathbbm{1}(i_{l}\leq p-1))}{\prod_{n=j-1}^{t-1}((\Delta+1)n+1)}.
\end{align}
Therefore, showing that \eqref{eqn:pdf} holds is equivalent to showing that:
\begin{align}\label{eqn:to_show}
&\prod\limits_{p=j}^{t}(p-1)(\Delta+1) + \sum\limits_{k=1}^{t-j+1}\hspace{-0.1cm}\sum\limits_{(i_{1},\cdots,i_{k})\in \mathcal{A}_{j,k}^{(t)}}\hspace{-0.2cm}\bigg(\hspace{-0.1cm}\prod\nolimits_{a=1}^{k}\big(1+(a-1)\Delta\big)\hspace{-0.8cm}\prod\limits_{p=j,p \notin \{i_{1},\cdots,i_{k}\}}^{t}\hspace{-0.8cm}\big((\hspace{0.03cm}p-1)(\Delta+1)-\Delta\sum\limits_{l=1}^{k}\mathbbm{1}(i_{l}\leq p-1)\big)\hspace{-0.1cm}\bigg)\nonumber\\
&\hspace{0.5cm}= \prod_{n=j-1}^{t-1}((\Delta+1)n+1).
\end{align}
 We prove \eqref{eqn:to_show} by induction on $t-j+1\ge 1$.

\medskip

\noindent
\underline{Base Case:} $t-j+1 = 1$ or $t=j$. For this case, the left-hand side of \eqref{eqn:to_show} is the following:
\begin{align*}
& \prod\limits_{p=j}^{j}(p-1)(\Delta+1)+ \sum\limits_{\mathcal{A}_{j,1}^{(j)}}\hspace{-0.1cm}\bigg(\hspace{-0.1cm}\prod\nolimits_{a=1}^{1}\big(1+(a-1)\Delta\big)\prod\limits_{p=j,p \hspace{0.02cm}\neq \hspace{0.02cm}j}^{j}\big((\hspace{0.03cm}p-1)(\Delta+1)-\Delta\sum\limits_{l=1}^{1}\mathbbm{1}(i_{l}\leq p-1)\big)\hspace{-0.1cm}\bigg)\\
&\hspace{0.5cm}= \hspace{0.02cm}(\hspace{0.08cm}j-1)(\Delta+1) + 1
\end{align*}
which, upon simplification and noting that the set $\mathcal{A}_{j,1}^{(j)}=\{j\}$, equals the right hand side of \eqref{eqn:to_show} for $t-j+1=1$. 

\medskip

\noindent
\underline{Induction Step:} 
We now show the induction step: assuming that \eqref{eqn:to_show} is true for $t-j+1 = s$, we show that it holds for $t-j+1=s+1$.
We thus assume that the following holds:
\begin{align}\label{eqn:assumption}
&\prod\limits_{p=j}^{j+s-1}((p-1)(\Delta+1)) +  \sum\limits_{k=1}^{s}\sum\limits_{\mathcal{A}_{j,k}^{(j+s-1)}}\bigg(\prod\limits_{a=1}^{k}\big(1+(a-1)\Delta\big)\hspace{-0.6cm}\prod\limits_{p=j,p \notin \{i_{1},\cdots,i_{k}\}}^{j+s-1}\hspace{-0.6cm}\big((\hspace{0.03cm}p-1)(\Delta+1)-\Delta\sum\limits_{l=1}^{k}\mathbbm{1}(i_{l}\leq p-1)\big)\bigg)\nonumber\\
&\hspace{0.5cm}= \prod_{n=j-1}^{j+s-2}((\Delta+1)n+1).
\end{align}
We next show the induction step using \eqref{eqn:assumption}, by starting from the right-hand side:
\begin{align*}
&\prod\limits_{n=j-1}^{j+s-1}((\Delta+1)n+1) = ((\Delta+1)(s+j-1)+1)\prod\limits_{n=j-1}^{j+s-2}((\Delta+1)n+1)\\
&\overset{(a)}{=} ((\Delta+1)(s+j-1)+1)\prod\limits_{p=j}^{j+s-1}(p-1)(\Delta+1) \\
& +\sum\limits_{k=1}^{s}((\Delta+1)(s+j-1)+1)\sum\limits_{\mathcal{A}_{j,k}^{(j+s-1)}}\bigg(\prod\limits_{a=1}^{k}\big(1+(a-1)\Delta\big)\hspace{-0.7cm}\prod\limits_{p=j,p \notin \{i_{1},\cdots,i_{k}\}}^{j+s-1}\hspace{-0.7cm}\big((\hspace{0.03cm}p-1)(\Delta+1)-\Delta\sum\limits_{l=1}^{k}\mathbbm{1}(i_{l}\leq p-1)\big)\bigg)\\
& \overset{(b)}{=} \prod_{p=j}^{j+s}(p-1)(\Delta+1) \quad + \quad \prod_{p=j}^{j+s-1}(p-1)(\Delta+1) \\ 
&+\sum\limits_{k=1}^{s}((\Delta+1)(s+j-1)-\Delta k + \Delta k+1)\hspace{-0.2cm}\sum\limits_{\mathcal{A}_{j,k}^{(j+s-1)}}\hspace{-0.2cm}\bigg(\hspace{-0.1cm}\prod\limits_{a=1}^{k}\hspace{-0.1cm}\big(1+(a-1)\Delta\big)\hspace{-0.8cm}\prod\limits_{p=j,p \notin \{i_{1},\cdots,i_{k}\}}^{j+s-1}\hspace{-0.8cm}\big((\hspace{0.03cm}p-1)(\Delta+1)-\Delta\sum\limits_{l=1}^{k}\mathbbm{1}(i_{l}\leq p-1)\big)\hspace{-0.1cm}\bigg)\\
& \overset{(c)}{=} \prod_{p=j}^{j+s}(p-1)(\Delta+1) \quad + \quad \prod_{p=j}^{j+s-1}(p-1)(\Delta+1)\\ 
&+\sum\limits_{k=1}^{s}((\Delta+1)(s+j-1)-\Delta k)\sum\limits_{\mathcal{A}_{j,k}^{(j+s-1)}}\hspace{-0.2cm}\bigg(\prod\limits_{a=1}^{k}\big(1+(a-1)\Delta\big)\hspace{-0.7cm}\prod\limits_{p=j,p \notin \{i_{1},\cdots,i_{k}\}}^{j+s-1}\hspace{-0.7cm}\big((\hspace{0.03cm}p-1)(\Delta+1)-\Delta\sum\limits_{l=1}^{k}\mathbbm{1}(i_{l}\leq p-1)\big)\hspace{-0.1cm}\bigg) \\
&+\sum\limits_{k=1}^{s}(\Delta k+1)\sum\limits_{\mathcal{A}_{j,k}^{(j+s-1)}}\hspace{-0.2cm}\bigg(\prod\limits_{a=1}^{k}\big(1+(a-1)\Delta\big)\hspace{-0.7cm}\prod\limits_{p=j,p \notin \{i_{1},\cdots,i_{k}\}}^{j+s-1}\hspace{-0.7cm}\big((\hspace{0.03cm}p-1)(\Delta+1)-\Delta\sum\limits_{l=1}^{k}\mathbbm{1}(i_{l}\leq p-1)\big)\hspace{-0.1cm}\bigg)\\
&\overset{(d)}{=}\prod_{p=j}^{j+s}(p-1)(\Delta+1) \quad + \quad \prod_{p=j}^{j+s-1}(p-1)(\Delta+1) \\ 
&+\sum\limits_{k=1}^{s}\sum\limits_{\mathcal{A}_{j,k}^{(j+s-1)}}\hspace{-0.2cm}\bigg(\prod\limits_{a=1}^{k}\big(1+(a-1)\Delta\big)\hspace{-0.7cm}\prod\limits_{p=j,p \notin \{i_{1},\cdots,i_{k}\}}^{j+s}\hspace{-0.7cm}\big((\hspace{0.03cm}p-1)(\Delta+1)-\Delta\sum\limits_{l=1}^{k}\mathbbm{1}(i_{l}\leq p-1)\big)\hspace{-0.1cm}\bigg)\\
&+\sum\limits_{k=1}^{s}\sum\limits_{\mathcal{A}_{j,k}^{(j+s-1)}}\hspace{-0.2cm}\bigg(\prod\limits_{a=1}^{k+1}\big(1+(a-1)\Delta\big)\hspace{-0.7cm}\prod\limits_{p=j,p \notin \{i_{1},\cdots,i_{k}\}}^{j+s-1}\hspace{-0.7cm}\big((\hspace{0.03cm}p-1)(\Delta+1)-\Delta\sum\limits_{l=1}^{k}\mathbbm{1}(i_{l}\leq p-1)\big)\hspace{-0.1cm}\bigg)\\
& \overset{(e)}{=}\prod_{p=j}^{j+s}(p-1)(\Delta+1) \quad + \quad \prod_{p=j}^{j+s-1}(p-1)(\Delta+1) \\ 
&+\sum\limits_{k=1}^{s}\sum\limits_{\mathcal{A}_{j,k}^{(j+s-1)}}\hspace{-0.2cm}\bigg(\prod\limits_{a=1}^{k}\big(1+(a-1)\Delta\big)\hspace{-0.7cm}\prod\limits_{p=j,p \notin \{i_{1},\cdots,i_{k}\}}^{j+s}\hspace{-0.7cm}\big((\hspace{0.03cm}p-1)(\Delta+1)-\Delta\sum\limits_{l=1}^{k}\mathbbm{1}(i_{l}\leq p-1)\big)\hspace{-0.1cm}\bigg)\\
&+\sum\limits_{k=2}^{s+1}\sum\limits_{\mathcal{A}_{j,k-1}^{(j+s-1)}}\hspace{-0.2cm}\bigg(\prod\limits_{a=1}^{k}\big(1+(a-1)\Delta\big)\hspace{-0.7cm}\prod\limits_{p=j,p \notin \{i_{1},\cdots,i_{k-1}\}}^{j+s-1}\hspace{-0.7cm}\big((\hspace{0.03cm}p-1)(\Delta+1)-\Delta\sum\limits_{l=1}^{k-1}\mathbbm{1}(i_{l}\leq p-1)\big)\hspace{-0.1cm}\bigg)\\
& \overset{(f)}{=}\prod_{p=j}^{j+s}(p-1)(\Delta+1) \quad + \quad \prod_{p=j}^{j+s-1}(p-1)(\Delta+1)\\ 
&+\sum\limits_{k=1}^{s}\sum\limits_{\mathcal{A}_{j,k}^{(j+s-1)}}\hspace{-0.2cm}\bigg(\prod\limits_{a=1}^{k}\big(1+(a-1)\Delta\big)\hspace{-0.7cm}\prod\limits_{p=j,p \notin \{i_{1},\cdots,i_{k}\}}^{j+s}\hspace{-0.7cm}\big((\hspace{0.03cm}p-1)(\Delta+1)-\Delta\sum\limits_{l=1}^{k}\mathbbm{1}(i_{l}\leq p-1)\big)\hspace{-0.1cm}\bigg)\\
&+\sum\limits_{k=2}^{s+1}\sum\limits_{\mathcal{B}_{j,k}^{(j+s)}}\bigg(\prod\limits_{a=1}^{k}\big(1+(a-1)\Delta\big)\hspace{-0.7cm}\prod\limits_{p=j,p \notin \{i_{1},\cdots,i_{k-1},j+s\}}^{j+s}\hspace{-0.3cm}\big((\hspace{0.03cm}p-1)(\Delta+1)-\Delta\sum\limits_{l=1}^{k-1}\mathbbm{1}(i_{l}\leq p-1)\big)\hspace{-0.1cm}\bigg)\\
&\overset{(g)}{=}\prod_{p=j}^{j+s}(p-1)(\Delta+1) \quad + \quad \prod_{p=j}^{j+s-1}(p-1)(\Delta+1)\\
&+\sum\limits_{\mathcal{A}_{j,1}^{(j+s-1)}}\hspace{-0.2cm}\bigg(\prod\limits_{a=1}^{1}\big(1+(a-1)\Delta\big)\hspace{-0.3cm}\prod\limits_{p=j,p \notin \{i_{1}\}}^{j+s}\hspace{-0.3cm}\big((\hspace{0.03cm}p-1)(\Delta+1)-\Delta\sum\limits_{l=1}^{1}\mathbbm{1}(i_{l}\leq p-1)\big)\hspace{-0.1cm}\bigg)\\
&+\sum\limits_{k=2}^{s}\sum\limits_{\mathcal{A}_{j,k}^{(j+s)}}\hspace{-0.1cm}\bigg(\prod\limits_{a=1}^{k}\big(1+(a-1)\Delta\big)\hspace{-0.7cm}\prod\limits_{p=j,p \notin \{i_{1},\cdots,i_{k}\}}^{j+s}\hspace{-0.7cm}\big((\hspace{0.03cm}p-1)(\Delta+1)-\Delta\sum\limits_{l=1}^{k}\mathbbm{1}(i_{l}\leq p-1)\big)\hspace{-0.1cm}\bigg)\quad + \quad\prod\limits_{a=1}^{s+1}(1+(a-1)\Delta)\\
& \overset{(h)}{=}\prod_{p=j}^{j+s}(p-1)(\Delta+1) + \sum\limits_{k=1}^{s+1}\sum\limits_{\mathcal{A}_{j,k}^{(j+s)}}\hspace{-0.1cm}\bigg(\prod\limits_{a=1}^{k}\big(1+(a-1)\Delta\big)\hspace{-0.7cm}\prod\limits_{p=j,p \notin \{i_{1},\cdots,i_{k}\}}^{j+s}\hspace{-0.7cm}\big((\hspace{0.03cm}p-1)(\Delta+1)-\Delta\sum\limits_{l=1}^{k}\mathbbm{1}(i_{l}\leq p-1)\big)\hspace{-0.1cm}\bigg).
\end{align*}
In the above set of equations, we obtain $(a)$ by substituting \eqref{eqn:assumption} in the left-hand side of $(a)$. In $(b)$, we add and subtract $\Delta k$ to the term $((\Delta+1)(s+j-1)+1)$ and split the summation across the terms $((\Delta+1)(s+j-1)-\Delta k)$ and $(\Delta k + 1)$ in $(c)$. In $(d)$, we absorb the terms $((\Delta k+1)(s+j-1)-\Delta k)$ and  $(\Delta k+1)$ into the product. We replace $k$ by $k-1$ in the fourth term on the left-hand side of $(e)$. We obtain the fourth term on the right-hand side of $(f)$ using  \eqref{eqn:setB}. On the right-hand side of $(g)$, the second term  can be written as follows:
\[\prod\limits_{p=j}^{j+s-1}(p-1)(\Delta+1) = \sum\limits_{\mathcal{B}_{j,1}^{(j+s)}}\bigg(\prod\limits_{a=1}^{1}\big(1+(a-1)\Delta\big)\hspace{-0.3cm}\prod\limits_{p=j,p \notin \{j+s\}}^{j+s}\hspace{-0.3cm}\big((\hspace{0.03cm}p-1)(\Delta+1)-\Delta\sum\limits_{l=1}^{0}\mathbbm{1}(i_{l}\leq p-1)\big)\hspace{-0.1cm}\bigg)\]
which is merged with the third term on the right-hand side of $(g)$ to obtain the $k=1$ term on the right-hand side of (h). Similarly, for the terms for $k=2$ to $k=s$, we merge both of the terms on the right-hand side of $(f)$ using \eqref{eqn:set} to obtain the fourth term on the right-hand side of $(g)$. The last term on the right-hand side of $(g)$ is evaluation of the fourth term on the right-hand side of $(f)$ at $k=s+1$. Finally $(h)$ is obtained by writing all the terms under one summation. Hence the proof follows from induction on $t-j+1$.
\end{proof}
\section{Simulation Results}\label{sec:simulations}
In this section, we present a comparative study\footnote{For details about the simulations, refer to the following link:\\
\url{https://drive.google.com/drive/folders/1uOmz4B6RQ0hRmuu_CTfb02jJ03B1SyEl?usp=share_link}} between our model and the Barab\'{a}si-Albert model in terms of following three features:
\begin{itemize}
\item Structural differences in small-sized graphs;
\item Degree distributions of the graphs obtained;
\item Expected birth time of vertices with a fixed degree.
\end{itemize}
We first illustrate the structural differences (in terms of vertex color allocation and degree)  between the graphs generated by our model and the Barab\'{a}si-Albert model. In the next set of simulations, we compare the degree distribution of both models by plotting the probability of randomly choosing a $k$ degree vertex versus $k$ (on a $\log-\log$ scale) for a graph generated until a fixed time instant. We give the degree distribution of graphs generated for $5000$ time steps (averaged over $250$ simulations) via the standard Barab\'{a}si-Albert model and our model with different choices of the reinforcement parameter $\Delta_{t}$ and discuss the similarities and differences obtained in the degree distributions. In the third set of simulations, we compare both models in terms of vertices expected birth time versus degree which we define as follows.

\begin{definition}
Given a random network/graph generated until time $t$, we define the vertices \textbf{expected birth time} for a fixed degree $k$, where $1\leq k \leq t$, as 
the expected value of all the times when the vertices which have degree $k$ at termination time $t$ were introduced. It is denoted by $\overline{b}_{t}(k)$ and is given by the following expression:
\begin{align}\label{eqn:exp_birthtime}
\overline{b}_{t}(k)= \sum\limits_{j=1}^{t} (j-1)p_{k}^{(j)}(t)   
\end{align}
where $p_{k}^{(j)}(t)$ is the probability that vertex $j$ has degree $k$ at termination time $t$. 
\end{definition}

Note that we write $(j-1)$ in~\eqref{eqn:exp_birthtime} because vertex $j$ is introduced in the network/graph at time $(j-1)$. In the experiments, we determine the empirical version of \eqref{eqn:exp_birthtime}, which we call the \textbf{average birth time} and denote it by $\mathrm{b}_{t}(k)$ for a degree $k$ and termination time $t$. It is given by

\begin{align}\label{eqn:avg_birthtime}
b_{t}(k) = \sum\limits_{j=1}^{t}(j-1) \, \frac{\mathbbm{1}\big(\sum\nolimits_{n=j}^{t}Z_{j,n} =k-1\big)}{\sum\limits_{j'=1}^{t}\mathbbm{1}\big(\sum\nolimits_{n=j'}^{t}Z_{j',n}=k-1\big)}.
\end{align}

Analysing the expected birth time for a random graph provides insights on the average ``age'' of vertices gaining a certain degree in the generated network. A high expected birth time for a high degree $k$ indicates that newer incoming vertices have (on average) accumulated more connections than older vertices. In the context of social networks, this scenario corresponds to {\em tardy influential} vertices that quickly gained more popularity over already existing ones despite being introduced later in the network. We will compute the empirical expected birth time (average birth time) for different choices of $\Delta_{t}$ and compare them with the Barab\'{a}si-Albert model later in this section.

In Figure \ref{fig:networks}, two $15$-vertex networks are depicted, one generated by our model (on the left-hand side) and the other by the Barab\'{a}si-Albert model (on the right-hand side). We make two observations: First, in contrast with the Barab\'{a}si-Albert model, in our model all vertices are labelled by distinct colors. This one-to-one correspondence between vertices and colors encodes all the information of the generated graph in the draw vectors of the underlying P\'{o}lya urn. Second, the maximum degree achieved is higher in the left-hand side network generated via our model (which achieves a maximum degree of $11$ compared to $6$ in the right-hand side Barab\'{a}si-Albert network). 
This happens along one sample path due to the choice of reinforcement parameter (here $\Delta_{t}=5$ for all $t\geq 1$) in our model which allows for an already selected vertex to be chosen with a higher probability than in the case of
the Barab\'{a}si-Albert model 
where the vertices are chosen proportional to their degree. In fact, this property is observed to persist in our simulations.
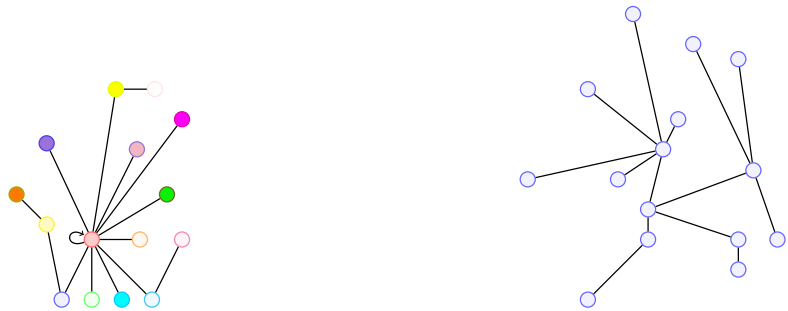
\begin{figure}[hbt]
\centering
\scalebox{0.4}{
\begin{tikzpicture}
[onenode/.style={circle, draw=red!60, fill=red!20, very thick, minimum size=5mm},
twonode/.style={circle, draw=green!60, fill=green!5, very thick, minimum size=5mm},
threenode/.style={circle, draw=blue!60, fill=blue!5, very thick, minimum size=5mm},
fournode/.style={circle, draw=orange!60, fill=orange!5, very thick, minimum size=5mm},
fivenode/.style={circle, draw=cyan!60, fill=cyan!5, very thick, minimum size=5mm},
sixnode/.style={circle, draw=magenta!60, fill=magenta!5, very thick, minimum size=5mm},
node/.style={circle, draw=cyan!60, fill=cyan!5, very thick, minimum size=5mm},
sevennode/.style={circle, draw=yellow!90, fill=yellow!30, very thick, minimum size=5mm},
eightnode/.style={circle, draw=blue!70!pink, fill= blue!40!pink, very thick, minimum size=5mm},
ninenode/.style={circle, draw=blue!40!pink, fill=blue!5!pink, very thick, minimum size=5mm},
tennode/.style={circle, draw=green!10!yellow, fill=green!2!yellow, very thick, minimum size=5mm},
elevennode/.style={circle, draw=pink!40, fill=pink!5, very thick, minimum size=5mm},
twelvenode/.style={circle, draw=purple!40!magenta, fill=purple!5!magenta, very thick, minimum size=5mm},
thirteennode/.style={circle, draw=green!40!orange, fill=purple!5!orange, very thick, minimum size=5mm},
fourteennode/.style={circle, draw=green!40!purple, fill=purple!6!green, very thick, minimum size=5mm},
fifteennode/.style={circle, draw=blue!20!cyan, fill=blue!2!cyan, very thick, minimum size=5mm},
]
\node[onenode] (1) at (0,4){};
\node[twonode] (2) at (0,2.0){};
\node[threenode] (3) at (-1.0,2.0){};
\node[fournode] (4) at (1.6,4.0){};
\node[fivenode] (5) at (2.0,2.0){}; 
\node[sixnode] (6) at (3,4) {};
\node[sevennode] (7) at (-1.5,4.5){};
\node[eightnode] (8) at (-1.5,7.2){}; 
\node[ninenode] (9) at (1.5,7.0){};
\node[tennode] (10) at (0.8,9.0){};
\node[elevennode] (11) at (2.1,9.0) {};
\node[twelvenode] (12) at (3.0,8.0){};
\node[thirteennode] (13) at (-2.5,5.5){};
\node[fourteennode] (14) at (2.5,5.5){};
\node[fifteennode] (15) at (1.0,2.0){};
    \path[-,very thick, black]
    (1) edge [out=200, in=150, loop] node{} (1)
    (1) edge (2)
    (1) edge  (3)
    (1) edge (4)
    (1) edge (5)
    (1) edge (8)
    (1) edge (9)
    (1) edge (10)
    (1) edge (12)
    (1) edge (14)
    (1) edge (15)
    (3) edge (7)

    (5) edge (6)
    (7) edge  (13)
    (10) edge (11);
\end{tikzpicture} 
}
\hspace{4.0cm}
\scalebox{0.4}{
\begin{tikzpicture}
[bluenode/.style={circle, draw=blue!60, fill=blue!5, very thick, minimum size=5mm},]
\node[bluenode] (1) at (6.0,4.0){};
\node[bluenode] (2) at (8.5,2.0){};
\node[bluenode] (3) at (9.0,3.0){};
\node[bluenode] (4) at (8.0,0.0){};
\node[bluenode] (5) at (11.0,-1.0){}; 
\node[bluenode] (11) at (6.0,-3.0) {};
\node[bluenode] (7) at (4.0,1.0){};
\node[bluenode] (8) at (11.0,-2.0){}; 
\node[bluenode] (9) at (7.5,6.5){};
\node[bluenode] (10) at (7.0,1.0){};
\node[bluenode] (6) at (8.0,-1.0) {};
\node[bluenode] (12) at (11.5,1.3){};
\node[bluenode] (13) at (12.3,-1.0){};
\node[bluenode] (14) at (11.0,5.0){};
\node[bluenode] (15) at (9.5,5.5){};
\tikzset{every loop/.style={}}
    \path[-,very thick, black]
    (1) edge(2)
    (2) edge (3)
    (2) edge (4)
    (2) edge (7)
    (2) edge (9)
    (2) edge (10)
    (4) edge (5)
    (4) edge (6)
    (4) edge (12)
    (5) edge (8)
    (6) edge (11)
    (12) edge (13)
    (12) edge (14)
    (12) edge (15);
    \end{tikzpicture}
}
 \caption{ On the left-hand side  is a $15$-vertex network generated via the draws from a P\'{o}lya urn with expanding colors and $\Delta_{t}=5$ for all $t \geq 1$ and on the right-hand side is a network with $15$ vertices generated via Barab\'{a}si-Albert model. For our model, unlike the Barab\'{a}si-Albert model, each vertex is represented by a distinct color which corresponds to a color type of balls in the P\'{o}lya urn at that time instant. Furthermore, the extra reinforcement parameter $\Delta_{t}$ in our model provides versatility in the level of preferential attachment. The parameter $\Delta_{t}=5$ in our model enables the central vertex of the graph on the left-hand side to obtain a higher degree ($11$ in this case) as compared to the right-hand side Barab\'{a}si-Albert network in which the highest degree achieved is $6$.}
\label{fig:networks}
\end{figure}
\begin{figure}[htp]
\begin{subfigure}[b]{0.49\textwidth}
\includegraphics[width=\textwidth]{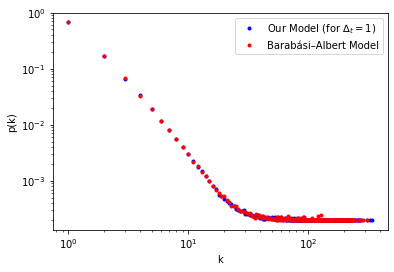}
\caption{$\Delta_{t}=1$}
\label{fig:degree_one}
\end{subfigure}
\hfill
\begin{subfigure}[b]{0.49\textwidth}
\includegraphics[width=\textwidth]{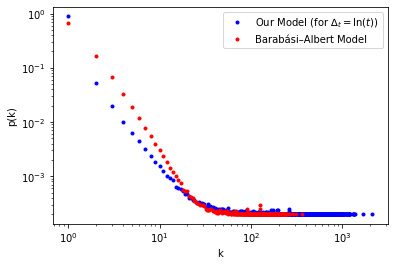}
\caption{$\Delta_{t}=\ln(t)$}
\label{fig:degree_ln}
\end{subfigure}
\caption{Degree distributions of networks generated until time $5000$ (averaged over $250$ simulations) for the Barab\'{a}si-Albert model and our model with $\Delta_{t}=1$ and $\Delta_{t}=\ln(t)$. In \textbf{(a)} the degree distributions of both models are nearly identical. While, in \textbf{(b)} the degree distributions are quite different.}
\label{fig:degree}
\end{figure}

In Figures \ref{fig:degree} and \ref{fig:degree_f(t)_g(t)}, we plot the degree distribution of networks generated for four different choices of $\Delta_{t}$: $1, \ln(t),f(t), g(t)$, where the functions $f(t)$ and $g(t)$ are defined in \eqref{eqn:fandg}. We observe the deviation of the degree distribution of the graphs generated via our model for the above mentioned choices of $\Delta_{t}$ from the degree distribution of the Barab\'{a}si-Albert network which follows the relation $p(k)\sim k^{-3}$, where $p(k)$ is the probability of randomly choosing a vertex of degree $k$ in the network. The similarity between the Barab\'{a}si-Albert algorithm and our model with $\Delta_{t}=1$ (as observed in Figure \ref{fig:degree} (a)) can be represented in the following way:
\begin{align*}
P(&\textrm{incoming vertex at time $t$ connects to vertex corresponding to color $c_{j}$} ) \\
&= \textrm{ratio of $c_{j}$ color balls in the expanding color P\'{o}lya urn at time $t-1$}\\
&= \frac{\textrm{degree of vertex corresponding to color $c_{j}$ in graph $\mathcal{G}_{t-1}$}}{\textrm{sum of degrees in graph $\mathcal{G}_{t-1}$}}\\
& = P(\textrm{incoming vertex connects to the vertex added at time $j-1$}\\
&\hspace{1.0cm} \textrm{in a standard Barab\'{a}si-Albert network}). 
\end{align*}
Hence in the case where $ \Delta_t=1$, the mechanisms of both models for iteratively constructing new vertices and edges are equivalent.
However the initialization of our model is different from the Barab\'{a}si-Albert model. In our model, the initial graph has only one vertex with a self-loop, whereas in the Barab\'{a}si-Albert model, the initial graph can potentially have more than one vertex equipped with an edge set and no self-loops. Even though the initialization of both models are different, the equivalent procedures for adding new vertices and edges between our model with $\Delta_{t}=1$ and the standard Barab\'{a}si-Albert model ensure that the generated graphs via both models will show similar properties for sufficiently large $t$. While it is an intricate task to analytically solve for the {\em asymptotic} degree distribution and other properties of our model due to the fact that its reinforcement dynamics is much more involved than that of the Barab\'{a}si-Albert model, such an investigation is a worthwhile future direction.

In Figure \ref{fig:degree} (b), we observe that the degree distribution of our model with $\Delta_{t}=\ln(t)$ significantly differs from the degree distribution of Barab\'{a}si-Albert model. The former has a lower probability of obtaining lower degree vertices (degree range $10^{0}-10^{1}$) as compared to the latter but has a slightly higher probability of gaining moderate degree vertices (degree range $50-150$). Additionally, the maximum degree attained in the case of $\Delta_{t}=\ln(t)$ for our model in Figure \ref{fig:degree} is much higher ($\sim 10^{3}$ as compared to only $200$ in Barab\'{a}si-Albert network).

\begin{figure}[htp]
\begin{subfigure}[b]{0.49\textwidth}
\includegraphics[width=\textwidth]{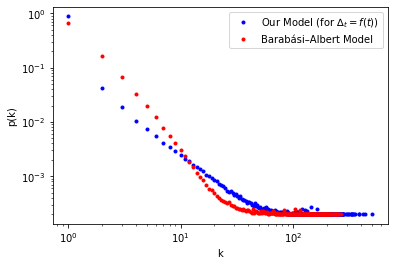}
\caption{$\Delta_{t}=f(t)$}
\end{subfigure}
\hfill
\begin{subfigure}[b]{0.49\textwidth}
\includegraphics[width=\textwidth]{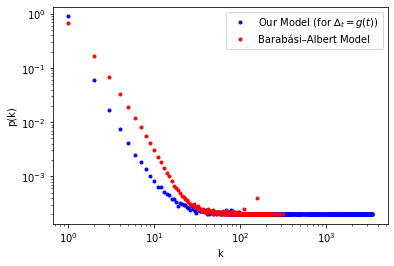}
\caption{$\Delta_{t}=g(t)$}
\end{subfigure}
\caption{Degree distribution of the Barab\'{a}si-Albert model and our model generated for two different choices of $\Delta_{t}$, \textbf{(a)}$\Delta_{t} = f(t)$ and \textbf{(b)}
$\Delta_{t}=g(t)$, where $f(t)$ and $g(t)$ are defined in \eqref{eqn:fandg}. Both plots are averaged over $250$ simulations, where each simulation is a generation of a $5000$-vertex graph.}
\label{fig:degree_f(t)_g(t)}
\end{figure}

In Figure \ref{fig:degree_f(t)_g(t)}, we present two more cases in which the degree distributions differ substantially between our model and the Barab\'{a}si-Albert model. More specifically, we generate our network for $\Delta_{t}$ being an increasing step function $f(t)$ and a decreasing continuous function $g(t)$ given by:
\begin{align}\label{eqn:fandg}
f(t) = \begin{cases}
1 &\textrm{for} \quad 0\leq t< 1000\\
10 &\textrm{for} \quad 1000\leq t<2500\\
100 &\textrm{for} \quad 2500\leq t \leq 5000,
\end{cases} \hspace{1cm} g(t) =  \begin{cases}
10 \hspace{1.0cm} &\textrm{for} \quad 0\leq t\leq 1000\\
\frac{10^{4}}{t} &\textrm{for} \quad 1000\leq t \leq 2000\\
5 &\textrm{for} \quad 2000\leq t \leq 3000\\
\frac{15\times 10^{3}}{t} &\textrm{for} \quad 3000\leq t \leq 4000\\
3.75 &\textrm{for} \quad 4000\leq t \leq 5000.
\end{cases}
\end{align}
We remark from Figure \ref{fig:degree_f(t)_g(t)} that the maximum degree attained in both figures (generated via $\Delta_{t}=f(t)$ and $\Delta_{t}=g(t)$) is higher than in the Barab\'{a}si-Albert model. Under the function $g(t)$, the constant value of $\Delta_{t}=10$ up until time $1000$ allows for the ball colors corresponding to older vertices (birth time $\leq 1000$) to get accumulated in large numbers. Thereafter, the value of $\Delta_{t}$ continuously decreases and therefore the ball colors corresponding to younger vertices (birth time $\geq 1000$) do not grow in large numbers. This gap between number of balls corresponding to younger and older vertices is significant enough for most of the older vertices to achieve high degrees with the younger vertices not acquiring much connections. Hence most of the vertices either get a very high degree or a very low degree resulting in less vertices with moderate  degrees ($10^{1}-10^{2}$) compared to the Barab\'{a}si-Albert network. On the contrary, the increasing step function $f$ allows for a wider range of ball colors (corresponding to vertices born between time $1000$ to $2500$) to grow their proportion, thus producing more vertices in the moderate degree range $(10^{1}-10^{2})$ compared to the Barab\'{a}si-Albert network.


In the next set of simulations, we compare  \eqref{eqn:avg_birthtime} for our network generated with $\Delta_{t}=1,\ln(t),f(t)$ and $g(t)$ and the Barab\'{a}si-Albert network.
\begin{figure}[htp]
\begin{subfigure}[b]{0.49\textwidth}
\includegraphics[width=\textwidth]{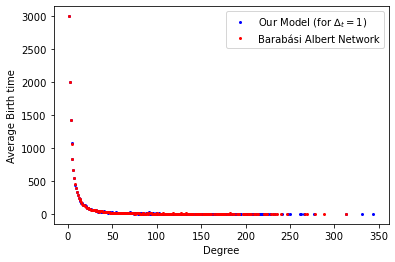}
\caption{$\Delta_{t}=1$}
\end{subfigure}
\hfill
\begin{subfigure}[b]{0.49\textwidth}
\includegraphics[width=\textwidth]{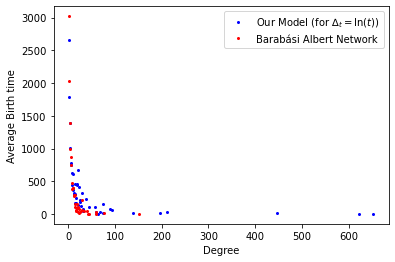}
\caption{$\Delta_{t}=\ln(t)$}
\end{subfigure}
\begin{subfigure}[b]{0.49\textwidth}
\vspace{0.5cm}
\includegraphics[width=\textwidth]{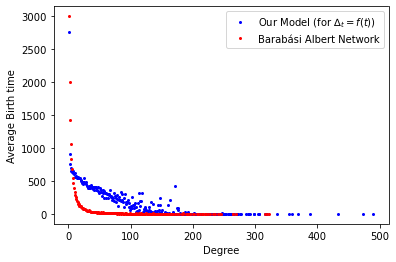}
\caption{$\Delta_{t}=f(t)$}
\end{subfigure}
\hfill
\begin{subfigure}[b]{0.49\textwidth}
\includegraphics[width=\textwidth]{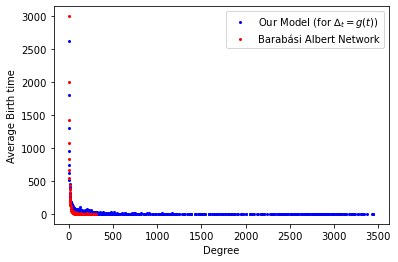}
\caption{$\Delta_{t}=g(t)$}
\end{subfigure}
\caption{Vertices average birth time versus degree for our model using \textbf{(a)} $\Delta_{t}=1$; \textbf{(b)}$\Delta_{t}=\ln(t)$; \textbf{(c)}$\Delta_{t}=f(t)$ and \textbf{(d)}$\Delta_{t}=g(t)$ (where the functions $f(t)$ and $g(t)$ are given in \eqref{eqn:fandg}) and for the Barab\'{a}si-Albert network. All networks are generated for $5000$ time steps and the average of 250 such networks is plotted.}
\label{fig:avg_birth_time}
\end{figure}
Figure \ref{fig:avg_birth_time}(a) demonstrates that in both the Barab\'{a}si-Albert model and our model for $\Delta_{t}=1$ vertices of the same degree are born at similar times. The stark similarities between the Barab\'{a}si-Albert model and our model for $\Delta_{t}=1$ in Figures \ref{fig:degree} and \ref{fig:avg_birth_time} strongly suggest that both networks have very similar structures; however a rigorous analytic study is required to confirm if our model with $\Delta_{t}=1$ is stochastically equivalent to the standard Barab\'{a}si-Albert model. In Figure \ref{fig:avg_birth_time}(b), we observe that for our model with $\Delta_{t}=\ln(t)$, the network shows slightly more connectivity in the vertices which are born at similar times as compared to the Barab\'{a}si-Albert network. This effect of same age vertices showing more connectivity when compared to Barab\'{a}si-Albert networks is much more amplified when our model uses $\Delta(t)=f(t)$ as shown in Figure \ref{fig:avg_birth_time}(c). Both cases provide a much richer algorithm for generating real-life networks in which the ``rich gets richer'' phenomenon needs to be dampened as it allows the more recently born vertices to get more connectivity. In contrast, Figure \ref{fig:avg_birth_time}(d) shows an amplification of the ``rich gets richer'' phenomenon when compared to the Barab\'{a}si-Albert network as the first two richest vertices achieve a significantly higher degree (around $3500$) compared to all other vertices. The rest of the vertices have very similar connectivity as that of the Barab\'{a}si-Albert network. The choice $\Delta_{t}=g(t)$ of the reinforcement parameter in our model provides an algorithm to generate graphs which are spatially similar to the Barab\'{a}si-Albert network but demonstrate a higher effect of preferential attachment. 
\section{Conclusion}\label{sec:conclusions}
We constructed a preferential attachment type graph using the draw vectors of a P\'{o}lya urn with growing colors and a tunable time-varying reinforcement parameter $\Delta_{t}$. The network obtained is essentially equivalent to the Barab\'{a}si-Albert network for the case $\Delta_{t}=1$ and  gains a significant amount of versatility when $\Delta_{t}$ is a time-varying function.  We analysed the draw vectors of the underlying stochastic process and derived the probability distribution of a random variable  counting the draws of a particular color of this P\'{o}lya process. This random variable can be written in terms of the degree of the vertex in the constructed preferential attachment network corresponding to this color. We provided simulation evidences for the structural similarities between our model and the Barab\'{a}si-Albert model for $\Delta_{t}=1$ and also justified the richness and versatility of our model for general $\Delta_{t}$. Future directions include devising a preferential attachment graph generating algorithm using a P\'{o}lya urn with finitely many colors, formulating strategies for choosing the best possible $\Delta_{t}$ for a randomly growing graph, incorporating removal of edges in the graph through removal of balls from the P\'{o}lya urn and setting an upper limit on the maximum degree a vertex can achieve.   
\section*{Acknowledgement}
This work was funded in part by the Natural Sciences and Engineering Research Council (NSERC) of Canada.

\bibliographystyle{IEEEtran}
\bibliography{example}
\end{spacing}

\end{document}